\documentclass[11pt,english,reqno]{amsart}

\usepackage{amsthm,amsmath,amssymb,amsfonts}
\usepackage{booktabs,float,caption,setspace}
\usepackage{mathtools,mleftright}
\usepackage[usenames]{color}
\usepackage{a4wide}

\definecolor{webgreen}{rgb}{0,.5,0}
\definecolor{webbrown}{rgb}{.6,0,0}

\usepackage[
  pdfauthor={},
  pdfkeywords={},
  pdftitle={},
  pdfcreator={},
  pdfproducer={},
  linktocpage,colorlinks,bookmarksnumbered,linkcolor=blue,
  citecolor=webgreen,urlcolor=webbrown]{hyperref}

\theoremstyle{plain}
\newtheorem{theorem}{Theorem}
\newtheorem{corollary}[theorem]{Corollary}

\newtheorem{lemma}[theorem]{Lemma}

\theoremstyle{definition}

\newtheorem{remark}[theorem]{Remark}
\newtheorem{conjecture}[theorem]{Conjecture}

\numberwithin{equation}{section}
\numberwithin{theorem}{section}
\numberwithin{figure}{section}
\numberwithin{table}{section}

\newcommand{\FF}{\mathbb{F}}
\newcommand{\ZZ}{\mathbb{Z}}
\newcommand{\andq}{\quad \text{and} \quad}
\newcommand{\textq}[1]{\quad \text{#1} \quad}
\newcommand{\norm}[1]{\mleft\|#1\mright\|}
\newcommand{\set}[1]{\mleft\{#1\mright\}}
\newcommand{\coeff}[1]{\mleft[#1\mright]}
\newcommand{\BE}{\mathcal{B}}
\newcommand{\PF}{\mathcal{P}}
\newcommand{\PFS}{\mathcal{P}_\Sigma}
\newcommand{\ST}{\mathcal{S}}
\newcommand{\WQ}{\mathcal{W}}
\newcommand{\AH}{\overline{H}}
\newcommand{\sigpi}{\widehat{\sigma}}
\newcommand{\seqnum}[1]{\href{https://oeis.org/#1}{\rm \underline{#1}}}

\makeatletter
\newcommand{\spod}[1]{\allowbreak\if@display\mkern8mu\else\mkern4mu\fi(#1)}
\newcommand{\smod}[1]{\spod{{\operator@font mod}\mkern6mu#1}}
\makeatother

\begin{document}

\title[Wilson's theorem modulo higher prime powers I]
{Wilson's theorem modulo higher prime powers I:\\Fermat and Wilson quotients}
\author{Bernd C. Kellner}
\address{G\"ottingen, Germany}
\email{bk@bernoulli.org}

\begin{abstract}
We show that Wilson's theorem as well as the Wilson quotient can be described by
supercongruences modulo any higher prime power involving terms of power sums of 
Fermat quotients. The new approach uses Bell polynomials and Newton's identities 
relating elementary symmetric polynomials to power sums. This enables us to compute 
certain multivariate polynomials recursively that are needed to establish the 
supercongruences. Subsequently, we give a recurrence formula for these polynomials 
and show further properties.
\end{abstract}

\maketitle


\section{Introduction}

Let $p$ be an odd prime throughout the paper. The well-known Wilson's theorem states that
\[
  (p-1)! \equiv -1 \smod{p},
\]
which can be proved in various ways. This leads to the definition of the Wilson quotient 
\begin{equation} \label{eq:wp}
  \WQ_p = \frac{(p-1)! + 1}{p}.
\end{equation}
By Fermat's little theorem, the congruence
\[
  a^{p-1} \equiv 1 \smod{p}
\]
holds for all integers $a$ coprime to $p$, which provides the definition of the Fermat quotient
\begin{equation} \label{eq:qp}
  q_p(a) = \frac{a^{p-1} - 1}{p}.
\end{equation}

In 1771, Lagrange~\cite{Lagrange:1771} gave a first proof of Wilson's theorem by the relation
\begin{equation} \label{eq:prod}
  \prod_{a=1}^{p-1} (x-a) \equiv x^{p-1} - 1 \smod{p}.
\end{equation}
(Note that the terms are written equivalently as $x+a$ in the original paper~\cite{Lagrange:1771}.) 
Simultane\-ously, the above congruence also provides a proof of Fermat's theorem.
Moreover, the relation gives the $p-1$ distinct roots of the polynomial $x^{p-1} - 1$ in 
$\FF_p^\times$, where $\FF_p$ denotes the finite field of $p$ elements.
For this reason, Bachmann~\cite[Chap.\,5, pp.\,153--179]{Bachmann:1902} devoted a joint chapter 
to the theorems of Fermat and Wilson, also giving a historical overview of the results in 1902.

A view years later, Lerch~\cite{Lerch:1905} established the essential connection in 1905 that
\begin{equation} \label{eq:wp-qp}
  \WQ_p \equiv \sum_{a=1}^{p-1} q_p(a) \smod{p},
\end{equation}
and showed several identities of the Fermat quotients. The basic logarithmic rule
\[
  q_p(ab) \equiv q_p(a) + q_p(b) \smod{p}
\]
was found by Eisenstein~\cite{Eisenstein:1850} before in 1850. 

We consider sums of powers of the Fermat quotients defined by
\begin{equation} \label{eq:Qp}
  Q_p(n) = \sum_{a=1}^{p-1} q_p(a)^n \quad (n \geq 1),
\end{equation}
in order to establish supercongruences, i.e., congruences modulo any higher prime power, 
of the Wilson quotient $\WQ_p$ and of the factorial $(p-1)!$. 
In a forthcoming paper \cite{Kellner:2025}, we will translate these results into 
congruences in terms of Bernoulli numbers. In this latter context, 
Glaisher~\cite{Glaisher:1900} derived a congruence of $(p-1)! \smod{p^2}$ in 1900. 
It took 100 years to achieve the next result $(p-1)! \smod{p^3}$ provided by Z.~H.~Sun~\cite{Sun:2000}. 
Both results are causally induced by considering the product of the left-hand side of \eqref{eq:prod}.

Our new approach uses the basic relationship between \eqref{eq:wp} and \eqref{eq:qp} 
by evaluating terms $\smod{p^n}$ for any $n \geq 1$, which can then be converted 
into a recursive procedure. We further use Bell polynomials and Newton's identities 
relating elementary symmetric polynomials to power sums. This leads to the definition 
of certain multivariate polynomials that can be recursively computed.

As a matter of fact, Lerch~\cite[pp.\,471--472]{Lerch:1905} handled only the simple 
case $\WQ_p \smod{p}$ to derive his congruence \eqref{eq:wp-qp} in a straightforward way. 
However, the general case could have been revealed for 120 years.
The main result of the paper is as follows.

\begin{theorem} \label{thm:main}
We have the following statements:
\begin{enumerate}
\item There exist unique multivariate polynomials 
\[
  \psi_\nu(x_1,\ldots,x_\nu) \in \ZZ[x_1,\ldots,x_\nu] \quad (\nu \geq 1),
\]
which have no constant term. These polynomials can be computed recursively; 

\item Let ${n \geq 1}$ and ${p > n}$ be an odd prime. Then we have 
\begin{align*}
  \WQ_p &\equiv \sum_{\nu=1}^{n} \frac{p^{\nu-1}}{\nu!} \, \psi_\nu( Q_p(1),\ldots, Q_p(\nu)) \smod{p^n}, \\
\shortintertext{and equivalently,}
  (p-1)! &\equiv -1 + \sum_{\nu=1}^{n} \frac{p^\nu}{\nu!} \, \psi_\nu( Q_p(1),\ldots, Q_p(\nu)) \smod{p^{n+1}}.
\end{align*}
\end{enumerate}
\end{theorem}

See Table~\ref{tbl:psi} for the first few computed polynomials $\psi_\nu$ 
and Table~\ref{tbl:psi2} for continued computations, respectively.
The recurrence formula for $\psi_\nu$ and some properties of these polynomials are 
presented in Section~\ref{sec:prop}, since we need to introduce further notation 
and definitions. 

\begin{corollary}
Let ${n \geq 1}$ and ${p > n}$ be an odd prime. For computational purposes, we need 
to compute the following initial terms and to evaluate the polynomials $\psi_\nu$ 
in different moduli. 
For $\WQ_p \smod{p^n}$ and $(p-1)! \smod{p^{n+1}}$, respectively, we have
\[
  \set{Q_p(1), \psi_1 \smod{p^n}, Q_p(2), \psi_2 \smod{p^{n-1}}, \ldots, Q_p(n), \psi_n \smod{p}}.
\]
\end{corollary}

\begin{table}[H] \small
\setstretch{1.25}
\begin{center}
\begin{tabular}{r@{\;=\;}l}
  \toprule
  $\psi_1$ & $x_1$ \\
  $\psi_2$ & $2 x_1 - x_1^2 - x_2$ \\ 
  $\psi_3$ & $6 x_1 - 6 x_1^2 + x_1^3 + 3 x_1 x_2 - 3 x_2 + 2 x_3$ \\
  $\psi_4$ & $24 x_1 - 36 x_1^2 + 12 x_1^3 - x_1^4 - 6 x_1^2 x_2 + 24 x_1 x_2 - 8 x_1 x_3 - 12 x_2 - 3 x_2^2 + 8 x_3 - 6 x_4$ \\
  \bottomrule
\end{tabular}

\caption{First few polynomials $\psi_\nu$.}
\label{tbl:psi}
\end{center}
\end{table}

Let $\PF(n)$ be the partition function for $n \geq 1$. Define the partial sums 
$\PFS(n) = \sum_{\nu=1}^{n} \PF(\nu)$.
For a polynomial $f$, let $\# f$ denote the number of its terms.

\begin{theorem} \label{thm:num}
For $n \geq 1$, we have that $\# \psi_n \leq \PFS(n)$.
\end{theorem}

The first few values of $\PFS$ are
\[
  1, 3, 6, 11, 18, 29, 44, 66, 96, 138, 194, 271, 372, 507, 683, 914, \ldots,
\]
which is sequence \seqnum{A026905} in OEIS~\cite{OEIS}.

We actually need the help of computer algebra systems for such calculations as given in Tables~\ref{tbl:psi} 
and~\ref{tbl:psi2}. We used \textsl{Mathematica} to compute the polynomials and related terms,
and to check all results of the paper. Note that further \textsl{improvements} to higher prime powers 
will only lead to an immense number of terms, which grow exponentially due to the partition function.

Since the terms of the polynomials $\psi_\nu$ have different signs and are determined recursively, 
it is not clear whether terms can vanish. However, computing the first 30 polynomials $\psi_\nu$ 
(note that $\# \psi_{30} = \PFS(30) = 28\,628$) and verifying the equality in this range, 
we may state the following conjecture.

\begin{conjecture} 
For $n \geq 1$, we have that $\# \psi_n = \PFS(n)$.
\end{conjecture}

The rest of the paper is organized as follows. The next section introduces the 
Bell polynomials and elementary symmetric polynomials. Section~\ref{sec:proof} 
contains the proof of the main Theorem~\ref{thm:main}.
In the last Section~\ref{sec:prop}, we present the recurrence formula for $\psi_\nu$ 
in terms of Bell polynomials and show further properties. 
This results in a proof of Theorem~\ref{thm:num}. Subsequently, we state a 
conjecture about the sum of the coefficients of the polynomials $\psi_\nu$.


\section{Bell polynomials and Newton's identities}

For $n \geq 1$ and $1 \leq k \leq n$, the partial Bell polynomials $\BE_{n,k}$ 
are homogeneous polynomials of degree $k$. They are defined by
\begin{equation} \label{eq:bnk}
  \BE_{n,k}(x_1,\ldots,x_{n-k+1}) 
  = \sum_{\substack{j_1+2j_2+3j_3+\cdots = n\\j_1+j_2+j_3+\cdots = k}} 
  \frac{n!}{j_1! \cdots j_{n-k+1}!} 
  \prod_{\nu=1}^{n-k+1} \left(\frac{x_\nu}{\nu!}\right)^{j_\nu},
\end{equation}
which have integral coefficients. Moreover, $\BE_{n,k}$ contains $\PF(n,k)$ monomials, 
where $\PF(n, k)$ is the number of partitions of $n$ into $k$ summands. 
See Bell~\cite{Bell:1934} and Comtet~\cite[Chaps. 2.1, 3.3, 6.6]{Comtet:1974}. 
The generating function reads
\begin{equation} \label{eq:gf-bnk}
  \frac{1}{k!} \mleft( \sum_{n \geq 1} x_n \frac{t^n}{n!} \mright)^{\!\!\!k}
  = \sum_{n \geq k} \BE_{n,k}(x_1,\ldots,x_{n-k+1}) \frac{t^n}{n!}.
\end{equation}

The complete Bell polynomials $\BE_n$ are given by
\begin{align*}
  \BE_n(x_1,\ldots,x_n) &= \sum_{k=1}^{n} \BE_{n,k}(x_1,\ldots,x_{n-k+1}),
\shortintertext{satisfying the generating function}
  \exp \mleft( \sum_{n \geq 1} x_n \frac{t^n}{n!} \mright)
  &= 1 + \sum_{n \geq 1} \BE_n(x_1,\ldots,x_n) \frac{t^n}{n!}.
\end{align*}

For $k \geq 1$, let $\PF_k$ be the set of partitions of $k$, and let $\PF$ contain 
all partitions. Write any partition $\gamma \in \PF_k$ as an ascending ordered tuple 
$\gamma = (\gamma_1,\ldots,\gamma_{\ell})$ of length 
$\ell = |\gamma|$ and $k = \norm{\gamma}$ being its sum.
We write a monomial as
\[
  x_\gamma = \prod_{\nu=1}^{|\gamma|} x_{\gamma_\nu}.
\] 
Let $f, g \in \ZZ[x_1,x_2,\ldots]$. Write the polynomial $f$ as a finite representation
\[
  f = \sum_{\gamma \in \PF} c_\gamma \, x_\gamma \textq{with}
  c_\gamma \in \ZZ \setminus \set{0},
\]
where an empty sum is defined to be $0$. Define the maximum partition order as
\[
  \norm{f} = \max \set{\norm{\gamma} : x_\gamma \text{ is a monomial of } f}
\]
and $\norm{0} = 0$, obeying the strong triangle inequality such that
\[
  \norm{f+g} \leq \max (\norm{f}, \norm{g}).
\]
For example, we obtain for \eqref{eq:bnk} that
\begin{equation}\label{eq:bnk-ord}
  \norm{\BE_{n,k}(x_1,\ldots,x_{n-k+1})} = n.
\end{equation}

The elementary symmetric polynomials $\sigma_\nu$ in $n$ variables are defined by
\[
  \sigma_\nu = \sigma_\nu(x_1,\ldots,x_n)
  = \sum_{\substack{J \subseteq \set{1,\ldots,n}\\ |J| = \nu}} \prod_{j \in J} x_j
  \quad (1 \leq \nu \leq n)
\]
with $\sigma_0 = 1$. This follows from the generating function
\[
  \prod_{j=1}^{n} (1 + x_j \, t) = 1 + \sum_{\nu=1}^{n} \sigma_\nu \, t^\nu.
\]

Let $\pi_\nu$ denote the power sums in $n$ variables such that
\[
  \pi_\nu = \pi_\nu(x_1,\ldots,x_n) = x_1^\nu + \cdots + x_n^\nu
  \quad (1 \leq \nu \leq n). 
\]
The Newton identities establish a connection between the elementary symmetric 
polynomials $\sigma_\nu$ and the power sums $\pi_\nu$. 
To indicate the change of variables, we use the notation
\[
  \sigpi_\nu = \sigpi_\nu(\pi_1,\ldots,\pi_\nu)
  \quad (1 \leq \nu \leq n).
\]
Then the equality holds that
\[
  \sigma_\nu = \sigpi_\nu \quad (1 \leq \nu \leq n).
\]

With the help of the Bell polynomials, one finally gets the expressions
\begin{align}
  \sigpi_k &= \frac{(-1)^k}{k!} \, \BE_k(-\pi_1,-1! \, \pi_2,\ldots,-(k-1)! \, \pi_k) \label{eq:sigma-bell} \\
  &= (-1)^k \! \sum_{j_1+2j_2+3j_3+\cdots = k} \! \frac{(-1)^{j_1+\cdots+j_k}}{j_1! \cdots j_k!}
  \prod_{\nu=1}^{k} \left(\frac{\pi_\nu}{\nu}\right)^{j_\nu} \nonumber
\end{align}
for $k \geq 1$, where the polynomials are independent of $n$ (see Table~\ref{tbl:sigma}).
The following lemma results from the above definitions (cf.~\cite{Comtet:1974}).

\begin{lemma} \label{lem:sigma-pi}
For $k \geq 1$, the polynomial $\sigpi_k^\star = k!\,\sigpi_k$ in terms of $\pi_\nu$ 
has integral coefficients, ${\# \sigpi_k = |\PF_k|}$, and $\norm{\sigpi_k^\star} = k$. More precisely,
\[
  \sigpi_k = \frac{1}{k!} \sum_{\gamma \in \PF_k}
  c_\gamma \prod_{\nu=1}^{|\gamma|} \pi_{\gamma_\nu}
\]
with coefficients $c_\gamma \in \ZZ \setminus \set{0}$. In particular, 
for $k \geq 2$ we have
\begin{equation} \label{eq:sigpi}
  \sigpi_k^\star = \pi_1^k + \cdots + (-1)^{k-1} (k-1)! \, \pi_k,
\end{equation}
and the sum of the coefficients vanishes, namely, 
\[
  \sum_{\gamma \in \PF_k} c_\gamma = 0.
\]
\end{lemma}

\begin{table}[H] \small
\setstretch{1.25}
\begin{center}
\begin{tabular}{r@{\;=\;}l}
  \toprule
  $\sigpi_1$ & $\pi_1$ \\
  $\sigpi_2$ & $\frac{1}{2}(\pi_1^2 - \pi_2)$ \\ 
  $\sigpi_3$ & $\frac{1}{3!}(\pi_1^3 - 3 \pi_1 \pi_2 + 2 \pi_3)$ \\
  $\sigpi_4$ & $\frac{1}{4!}(\pi_1^4 - 6 \pi_1^2 \pi_2 + 8 \pi_1 \pi_3 + 3 \pi_2^2 - 6 \pi_4)$ \\
  $\sigpi_5$ & $\frac{1}{5!}(\pi_1^5 - 10 \pi_1^3 \pi_2 + 20 \pi_1^2 \pi_3 + 15 \pi_1 \pi_2^2 - 30 \pi_1 \pi_4 - 20 \pi_2 \pi_3 + 24 \pi_5)$ \\
  \bottomrule
\end{tabular}

\caption{First few polynomials $\sigpi_k$ in terms of $\pi_\nu$.}
\label{tbl:sigma}
\end{center}
\end{table}


\section{Proof of the main theorem}
\label{sec:proof}

Recall $Q_p$ in \eqref{eq:Qp} as the power sums of $q_p$. We use the notation 
\begin{align}
  \sigma_\nu(q_p) &= \sigma_\nu( q_p(1), \ldots, q_p(p-1)), \nonumber \\
  \sigpi_\nu(Q_p) &= \sigpi_\nu( Q_p(1), \ldots, , Q_p(\nu)) \label{eq:sigpi-Qp}
\end{align}
for the elementary symmetric and power sum polynomials, respectively. 
We need the following lemmas and theorems to give a proof of Theorem~\ref{thm:main}
at the end of this section.

\begin{lemma}
Let $p$ be an odd prime. Then we have 
\[
  \prod_{a=1}^{p-1} (1 + p \, q_p(a)) = (1 - p \WQ_p)^{p-1},
\]
which gives the expansions
\begin{equation} \label{eq:wp-sigma}
  \sum_{\nu=0}^{p-1} p^\nu \sigma_\nu(q_p)
  = \sum_{\nu=0}^{p-1} \binom{p-1}{\nu} (-1)^\nu p^\nu \, \WQ_p^\nu.
\end{equation}
\end{lemma}

\begin{proof}
Note that $p-1$ is even. Expanding the product $\prod_{a=1}^{p-1} a^{p-1} = (p-1)!^{p-1}$
in conjunction with \eqref{eq:wp} and \eqref{eq:qp} provides the desired products
and their expansions.
\end{proof}

Let $\ZZ_p$ be the ring of $p$-adic integers. We consider the $p$-adic expansion 
\[
  a = \alpha_0 + \alpha_1 \, p + \alpha_2 \, p^2 + \cdots
\]
with prescribed $\alpha_\nu \in \ZZ_p$ for $\nu \geq 0$, where the $\alpha_\nu$ are 
given by algebraic expressions. Define the linear operator $\coeff{p^\ell}$ giving 
the expression at $p^\ell$ such that $\coeff{p^\ell} \, a = \alpha_\ell$. 

\begin{theorem}
Let $n \geq 2$ and $p > n$ be an odd prime. Compute
\begin{equation} \label{eq:wp-1}
  \WQ_{p,1} \equiv Q_p(1) \smod{p},
\end{equation}
and iteratively for $\ell=2,\ldots,n$, compute
\begin{equation} \label{eq:wp-ell}
  \WQ_{p,\ell} \equiv Q_p(1) + p \WQ_{p,\ell-1} +
  \sum_{\nu=1}^{\ell-1} p^\nu \mleft( \sigpi_{\nu+1}(Q_p)
  + \binom{p-1}{\nu+1} (-1)^{\nu} \, \WQ_{p,\ell-\nu}^{\nu+1} \mright) \smod{p^\ell}.
\end{equation}
Then we have
\[
  \WQ_p \equiv \WQ_{p,n} \smod{p^n}.
\]
\end{theorem}

\begin{proof}
We rewrite \eqref{eq:wp-sigma} as follows. Remove the constant term $1$ for 
${\nu = 0}$, divide by $p$, and shift the index $\nu \mapsto \nu+1$ on both sides.
Since $p > n$, we arrive at the congruence
\[
  \sum_{\nu=0}^{n-1} p^\nu \sigma_{\nu+1}(q_p)
  \equiv \sum_{\nu=0}^{n-1} \binom{p-1}{\nu+1} (-1)^{\nu+1} p^\nu \, \WQ_p^{\nu+1} \smod{p^n}.
\]
We have the identity $\sigma_\nu(q_p) = \sigpi_\nu(Q_p)$.
After some rearranging of terms, we derive that
\begin{equation} \label{eq:wp-n}
  \WQ_p \equiv Q_p(1) + p \WQ_p +
  \sum_{\nu=1}^{n-1} p^\nu \mleft( \sigpi_{\nu+1}(Q_p)
  + \binom{p-1}{\nu+1} (-1)^{\nu} \, \WQ_p^{\nu+1} \mright) \smod{p^n}.
\end{equation}
In the context of the above congruence, we set
\[
  \WQ_{p,\ell} \equiv \WQ_p \smod{p^\ell}
\]
for $\ell = 1, \ldots, n$. For $\ell = 1$, we obtain~\eqref{eq:wp-1}, 
which corresponds to Lerch's congruence~\eqref{eq:wp-qp}. Note that 
\[
  p^\nu \, \WQ_p \equiv p^\nu \, \WQ_{p,\ell-\nu} \smod{p^\ell}.
\]
For each step $\ell = 2,\ldots,n$, we can iteratively substitute such terms 
of $\WQ_p$ in this context with $\WQ_{p,\ell-\nu}$, being computed before, 
on the right-hand side of \eqref{eq:wp-n}. 
This finally leads to \eqref{eq:wp-ell} as desired.
\end{proof}

For $\nu \geq 1$, let 
\[
  \psi_\nu = \psi_\nu(x_1,\ldots,x_\nu) \in \ZZ[x_1,\ldots,x_\nu] 
\]
be multivariate polynomials. Similar to \eqref{eq:sigpi-Qp}, we write $\psi_\nu(Q_p)$.
Let $(n)_{\nu}$ denote the falling factorial such that $\binom{n}{\nu} = (n)_{\nu} / \nu!$.

\begin{lemma} \label{lem:p-sum-pow}
Let $n \geq k \geq 1$ and $p > n$ be an odd prime. Set $m=n-k+1$ and let $0 \leq r \leq k$.
For $k+r \leq \ell \leq n+r$, we have the identity
\[
  \coeff{p^\ell} p^r \frac{n!}{k!}
  \mleft( \sum_{\nu=1}^{m} \frac{p^\nu}{\nu!} \, \psi_\nu \mright)^{\!\!k}
  = \frac{n!}{(\ell-r)!} \, \BE_{\ell-r,k}(\psi_1,\ldots,\psi_{\ell-r-k+1})
\]
with integral coefficients, which vanishes for $0 \leq \ell < k+r$.
\end{lemma}

\begin{proof}
We set $y_\nu = \psi_\nu$ for $\nu = 1,\ldots,m$, and $y_\nu = 0$ otherwise. 
For $k+r \leq \ell \leq n+r$, we then infer from \eqref{eq:gf-bnk} that
\[
  \coeff{p^\ell} p^r \frac{n!}{k!}
  \mleft( \sum_{\nu \geq 1} \frac{p^\nu}{\nu!} \, y_\nu \mright)^{\!\!k}
  = \frac{n!}{(\ell-r)!} \, \BE_{\ell-r,k}(y_1,\ldots,y_{\ell-r-k+1}),
\]
having integral coefficients. Since $\ell -r \leq n$, so $\ell-r-k+1 \leq m$,
we have $y_\nu = \psi_\nu$ on the right-hand side above. 
For ${0 \leq \ell < k+r}$, the terms, shifted by $p^r$, 
vanish by the right-hand side of~\eqref{eq:gf-bnk}.
\end{proof}

\begin{theorem} \label{thm:wpl-psi}
Let $n \geq 2$ and $p > n$ be an odd prime. For $\ell = 1, \ldots, n$, we have
\begin{equation} \label{eq:wpl-psi}
  \WQ_{p,\ell} \equiv \sum_{\nu=1}^{\ell} \frac{p^{\nu-1}}{\nu!} \, \psi_\nu(Q_p) \smod{p^\ell},
\end{equation}
where $\psi_1 = x_1$ and recursively for $\nu = 2, \ldots, n$,
\begin{equation} \label{eq:psi-rec}
  \psi_\nu = \nu \, \psi_{\nu-1} + \sigpi_{\nu}^\star + \text{terms of } \psi_1, \ldots, \psi_{\nu-1},
\end{equation}
which have no constant term.
\end{theorem}

\begin{proof}
We use proof by induction. By \eqref{eq:wp-1}, we infer for $\ell = 1$ that
\[
  \WQ_{p,1} \equiv \psi_1(Q_p) \smod{p} \textq{with} \psi_1 = x_1. 
\]
Let $\ell \in \set{2,\ldots,n}$ and assume that \eqref{eq:wpl-psi} holds for 
$\ell-1, \ldots, 1$. From \eqref{eq:wp-ell}, it follows that
\[
  \WQ_{p,\ell} \equiv \psi_1(Q_p) + p \WQ_{p,\ell-1} +
  \sum_{\nu=1}^{\ell-1} p^\nu \mleft( \sigpi_{\nu+1}(Q_p)
  + \binom{p-1}{\nu+1} (-1)^{\nu} \, \WQ_{p,\ell-\nu}^{\nu+1} \mright) \smod{p^\ell}.
\]
We substitute the terms $\WQ_{p,\ell-\nu}$ for $\nu \geq 1$ by \eqref{eq:wpl-psi}. 
After some rewriting, we thus obtain

\begin{align}
  \WQ_{p,\ell} &\equiv\; \psi_1(Q_p)
  + \sum_{\nu=1}^{\ell-1} p^\nu \mleft( \frac{\psi_\nu(Q_p)}{\nu!} + \frac{\sigpi_{\nu+1}^\star(Q_p)}{(\nu+1)!}
  + S_{p,\ell,\nu} \mright) \smod{p^\ell}, \label{eq:wp-s1} \\
\shortintertext{where}
  S_{p,\ell,\nu} &\equiv (-1)^{\nu} \frac{(p-1)_{\nu+1}}{(\nu+1)!}
  \mleft( \sum_{j=1}^{\ell-\nu} \frac{p^{j-1}}{j!} \, \psi_j(Q_p) \mright)^{\!\!\nu+1} \!\! \smod{p^{\ell-\nu}}.
  \label{eq:wp-s2}
\end{align}

By assumption, we have
\[
  \WQ_{p,\ell} \equiv \sum_{\nu=1}^{\ell-1} \frac{p^{\nu-1}}{\nu!} \, \psi_\nu(Q_p) \smod{p^{\ell-1}}.
\]
Therefore, we have to collect terms, denoted as $T_{p,\ell}$, in context of $p^{\ell-1}$ such that
\[
  \WQ_{p,\ell} \equiv \sum_{\nu=1}^{\ell-1} \frac{p^{\nu-1}}{\nu!} \, \psi_\nu(Q_p)
  + \frac{p^{\ell-1}}{\ell!} T_{p,\ell} \smod{p^{\ell}},
\]
while higher terms with $p^{\ell+j}$ for $j \geq 0$ vanish.
Considering \eqref{eq:wp-s1} and \eqref{eq:wp-s2}, we further have to pick out terms 
of $\coeff{p^{\ell-1-\nu}} S_{p,\ell,\nu}$, which give a contribution to $T_{p,\ell}$. 
We then deduce that
\begin{equation} \label{eq:tpl}
  T_{p,\ell} \equiv \ell \, \psi_{\ell-1}(Q_p) + \sigpi_{\ell}^\star(Q_p) 
  + \ell! \sum_{\nu=1}^{\ell-1} \coeff{p^{\ell-1-\nu}} S_{p,\ell,\nu} \smod{p}.
\end{equation}

Now, we evaluate the terms involving $S_{p,\ell,\nu}$.
The case $\nu=\ell-1$ easily reduces to
\[
  \ell! \coeff{p^0} S_{p,\ell,\ell-1} \equiv -\ell! \, \psi_1^\ell(Q_p) \smod{p}.
\]
For the other cases, we shall simplify notation. Therefore, by shifting the index 
$\nu \mapsto \nu-1$, we need to handle $\nu = 2, \ldots, \ell-1$, as follows. 
Fix $\nu$ and set $m=\ell-\nu+1$. 
\begin{equation} \label{eq:spl}
  \ell! \coeff{p^{\ell-\nu}} S_{p,\ell,\nu-1}
  \equiv \ell! \coeff{p^{\ell}} p^\nu S_{p,\ell,\nu-1}
  \equiv \coeff{p^{\ell}} (-1)^{\nu-1} (p-1)_\nu \frac{\ell!}{\nu!} 
  \mleft( \sum_{j=1}^{m} \frac{p^j}{j!} \, \psi_j(Q_p) \mright)^{\!\!\nu}
  \!\! \smod{p}.
\end{equation}
Applying Lemma~\ref{lem:p-sum-pow}, we conclude that the above expression
has integral coefficients and depends on $\psi_1(Q_p), \ldots, \psi_{\ell-1}(Q_p)$.
Combining with \eqref{eq:tpl}, this shows that 
\[
  T_{p,\ell} \equiv \psi_\ell(Q_p) \smod{p}
\]
with some $\psi_\ell \in \ZZ[x_1,\ldots,x_\ell]$.
The determination of $\psi_\ell$ is independent of $p$ and $Q_p$.
Since congruence \eqref{eq:tpl} holds for all and infinitely many $p > n$, 
so it also holds in $\ZZ$ such that
\[
  \psi_\ell = \ell \, \psi_{\ell-1} + \sigpi_{\ell}^\star + \text{terms of } \psi_1, \ldots, \psi_{\ell-1}.
\]
By construction, $\psi_\ell$ has no constant term.
This shows the induction and completes the proof.
\end{proof}

We are now ready to give a proof of the main theorem. Note that
an exact recurrence formula is given by Theorem~\ref{thm:psi-rec} below.

\begin{proof}[Proof of Theorem~\ref{thm:main}]
By Theorem~\ref{thm:wpl-psi} and its proof, the polynomials $\psi_\nu$ for $\nu \geq 1$ 
can be determined independently of $p$ and $Q_p$, and they have a recurrence relation 
given by~\eqref{eq:psi-rec}. For ${n \geq 1}$ and ${p > n}$ an odd prime, the congruence 
of $\WQ_p \smod{p^n}$ follows from~\eqref{eq:wpl-psi}.
Equivalently, by \eqref{eq:wp} we obtain the congruence of $(p-1)! \smod{p^{n+1}}$.
\end{proof}


\section{Properties of the multivariate polynomials}
\label{sec:prop}

Recall the definitions of the former sections.
For $n \geq 1$ and $1 \leq k \leq n$, let $\ST_1(n,k)$ and  $\ST_2(n,k)$ be 
the Stirling number of the first and second kind, respectively.
These numbers are defined by
\[
  \sum_{k=1}^{n} \ST_1(n,k) \, x^k = (x)_n \andq
  \sum_{k=1}^{n} \ST_2(n,k) \, (x)_k = x^n,
\]
and also expressible by Bell polynomials via
\begin{align*}
  \ST_1(n,k) &= (-1)^{n-k} \BE_{n,k}(0!, 1!, \ldots, (n-k)!), \\
  \ST_2(n,k) &= \BE_{n,k}(1, 1, \ldots, 1). \\
\end{align*}

\begin{theorem} \label{thm:psi-rec}
For $n \geq 1$, we have the recurrence formula
\begin{equation} \label{eq:psi-rec2}
  \psi_n = n \, \psi_{n-1} + \sigpi_{n}^\star + \Psi_n,
\end{equation}
where $\psi_0 = 0$, 
\[
  \sigpi_{n}^\star = (-1)^n \BE_n(-x_1,-1! \, x_2,\ldots,-(n-1)! \, x_n),
\]  
and 
\begin{equation} \label{eq:psi-sum}
  \Psi_n = \sum_{\nu=2}^{n} \sum_{k=0}^{\min(\nu,n-\nu)} (-1)^{\nu+1} S_1(\nu+1,k+1) \,
  (n)_k \, \BE_{n-k,\nu}(\psi_1,\ldots,\psi_{n-k-\nu+1}).
\end{equation}
\end{theorem}

\begin{proof}
For $n=1$, \eqref{eq:psi-rec2} holds by $\psi_0 = \Psi_1 = 0$ and $\psi_1 = \sigpi_1^\star = x_1$. 
By Theorem~\ref{thm:wpl-psi} and \eqref{eq:psi-rec}, we have that \eqref{eq:psi-rec2} holds with 
${\psi_1 = x_1}$, $\sigpi_{n}^\star$ is given by \eqref{eq:sigma-bell}, 
and $\Psi_n =$ terms of $\psi_1, \ldots, \psi_{n-1}$ for $n \geq 2$.
We now follow the proof of Theorem~\ref{thm:wpl-psi}.
From \eqref{eq:tpl} and \eqref{eq:spl}, and translating the congruences into 
relations over polynomials with a similar notation, we infer that
\begin{equation} \label{eq:psi-sum2}
  \Psi_n = n! \sum_{\nu=2}^{n} \coeff{p^{n-\nu}} \widetilde{S}_{p,n,\nu-1},
\end{equation}
where for fixed $\nu = 2, \ldots, n$ and $m=n-\nu+1$, we have for each summand that
\[
  n! \coeff{p^{n-\nu}} \widetilde{S}_{p,n,\nu-1} = \coeff{p^{n}} (-1)^{\nu-1} (p-1)_\nu \frac{n!}{\nu!}
  \mleft( \sum_{j=1}^{m} \frac{p^j}{j!} \, \psi_j \mright)^{\!\!\nu}.
\]
By definition, we have the expansion
\[
  (p-1)_\nu = \sum_{k=0}^{\nu} S_1(\nu+1,k+1) \, p^k.
\]
We then obtain
\begin{align*}
  n! \coeff{p^{n-\nu}} \widetilde{S}_{p,n,\nu-1}
  &= \sum_{k=0}^{\nu} (-1)^{\nu-1} S_1(\nu+1,k+1) \coeff{p^{n}} p^k \frac{n!}{\nu!}
  \mleft( \sum_{j=1}^{m} \frac{p^j}{j!} \, \psi_j \mright)^{\!\!\nu} \\
  &= \sum_{k=0}^{\min(\nu,n-\nu)} (-1)^{\nu-1} S_1(\nu+1,k+1)
  \frac{n!}{(n-k)!} \BE_{n-k,\nu}(\psi_1,\ldots,\psi_{n-k-\nu+1}).
\end{align*}
The latter equation follows from Lemma~\ref{lem:p-sum-pow},
where the summation is bounded, since terms for $k+\nu > n$ vanish.
By summing the latter sum over $\nu$, \eqref{eq:psi-sum2} finally turns into \eqref{eq:psi-sum}
using the substitutions ${(-1)^{\nu+1} = (-1)^{\nu-1}}$ and $(n)_k = n!/(n-k)!$.
\end{proof}

See Tables~\ref{tbl:Psi} and~\ref{tbl:Psi2} for the first few computed polynomials~$\Psi_j$.
Unfolding the recurrence \eqref{eq:psi-rec2} immediately leads to the following result. 

\begin{corollary} \label{cor:psi-sum}
For $n \geq 1$, we have
\begin{align} 
  \psi_n &= \sum_{j=1}^{n} (n)_{n-j} \mleft( \sigpi_{j}^\star + \Psi_j \mright). \label{eq:psi-rec3}
\shortintertext{Moreover, for the sum of the coefficients of the polynomials, it follows that}
  \psi_n(1,\ldots,1) &= n! + \sum_{j=2}^{n} (n)_{n-j} \, \Psi_j(1,\ldots,1). \label{eq:psi-rec4}
\end{align}
\end{corollary}

\begin{proof}
As before, we have ${\psi_0 = \Psi_1 = 0}$ and ${\psi_1 = \sigpi_1^\star = x_1}$. 
The sum in \eqref{eq:psi-rec3} follows from unfolding the term $n \, \psi_{n-1}$ in \eqref{eq:psi-rec2}.
By Lemma~\ref{lem:sigma-pi}, we have that $\sigpi_{j}^\star(1,\ldots,1) = 0$ for $j \geq 2$.
The result~\eqref{eq:psi-rec4} then follows.
\end{proof}

\begin{theorem} \label{thm:psi-ord}
For $n \geq 1$ and $1 \leq k \leq n$, we have
\begin{equation} \label{eq:psi-ord}
  \norm{\psi_n} = n, \quad \norm{\Psi_n} \leq n, \andq
  \norm{\BE_{n,k}(\psi_1,\ldots,\psi_{n-k+1})} \leq n.
\end{equation}
In particular, we have $\BE_{n,n}(\psi_1) = x_1^n$ and $\norm{\BE_{n,n}(\psi_1)} = n$.
Moreover, $\sigpi_{n}^\star$ only gives a contribution of the variable $x_n$ to $\psi_n$.
More precisely, we have for $n \geq 2$ the pattern that
\begin{equation} \label{eq:psi-pat}
  \psi_n = n! \, x_1 + \cdots + (-1)^{n-1} (n-1)! \, x_n.
\end{equation}
\end{theorem}

\begin{proof}
We use proof by induction. The case ${n=1}$ trivially holds by ${\psi_0 = \Psi_1 = 0}$ 
and $\psi_1 = \sigpi_1^\star = x_1$, and
$\BE_{1,1}(\psi_1) = \psi_1$. Now, let ${n \geq 2}$ and assume that \eqref{eq:psi-ord} 
holds for $n-1, \ldots, 1$, and \eqref{eq:psi-pat} holds for $n-1$.
For ${1 \leq k \leq n}$, we derive from~\eqref{eq:bnk} that
\begin{equation} \label{eq:bnk-psi}
  \BE_{n,k}(\psi_1,\ldots,\psi_{n-k+1}) 
  = \sum_{\substack{j_1+2j_2+3j_3+\cdots = n\\j_1+j_2+j_3+\cdots = k}} 
  \frac{n!}{j_1! \cdots j_{n-k+1}!} 
  \prod_{\nu=1}^{n-k+1} \left(\frac{\psi_\nu}{\nu!}\right)^{j_\nu}.
\end{equation}
We first consider the case $k = 2,\ldots,n$, since only $\psi_1,\ldots,\psi_{n-1}$ are involved.
Fix one summand and index $\nu$ of the product of the right-hand side of \eqref{eq:bnk-psi}.
We look at these monomials, being a part of the product, and check their partitions. 
For example, in the simple case of 
\eqref{eq:bnk} and \eqref{eq:bnk-ord}, we would obtain
\[
  x_\nu^{j_\nu} = x_\gamma \textq{with} \norm{\gamma} = \nu j_\nu.
\]
Returning to \eqref{eq:bnk-psi}, we have a product of polynomials, namely, $\psi_\nu^{j_\nu}$.
We have to multiply these polynomials out. We choose any $j_\nu$ monomials from $\psi_\nu$. 
Then we obtain a product like
\[
  x_{\gamma'_1} \cdots x_{\gamma'_{j_\nu}} = x_{\gamma'}.
\]
with partitions $\gamma'_1, \ldots, \gamma'_{j_\nu}$, and $\gamma'$.
From $\norm{\psi_\nu} = \nu$ by assumption, we infer that
$\norm{\gamma'_\mu} \leq \nu$ for $\mu = 1, \ldots, j_\nu$
and so $\norm{\gamma'} \leq \nu j_\nu$.
Since this reasoning holds for all monomials, we conclude that
\begin{equation} \label{eq:bnk-leq}
  \norm{\BE_{n,k}(\psi_1,\ldots,\psi_{n-k+1})} \leq n.
\end{equation}
For $k=n$, we obtain by \eqref{eq:bnk-psi} the simple case that 
$\BE_{n,n}(\psi_1) = x_1^n$ and so $\norm{\BE_{n,n}(\psi_1)} = n$.

Regarding $\Psi_n$ and \eqref{eq:psi-sum}, we have to consider terms of 
$\BE_{n-k,\nu}$ with ${k \geq 0}$ and ${\nu \geq 2}$.
Therefore, it follows from \eqref{eq:bnk-leq} that $\norm{\Psi_n} \leq n$. 
Moreover, by \eqref{eq:psi-sum} and \eqref{eq:bnk-psi}, 
the monomials $x_1$ and $x_n$ cannot occur in $\Psi_n$.
Since $\norm{\psi_{n-1}} = n-1$ by assumption, it follows from \eqref{eq:psi-rec2} that 
$\sigpi_{n}^\star$, having monomials $x_\gamma$ for all partitions $\gamma$ of $n$ 
by Lemma~\ref{lem:sigma-pi}, 
can only contribute the monomial $x_n$ to $\psi_n$, showing that $\norm{\psi_n} = n$.
By \eqref{eq:sigpi}, this is the term $(-1)^{n-1} (n-1)! \, x_n$
as claimed in \eqref{eq:psi-pat}. Since the monomial $x_1$ is not in $\sigpi_{n}^\star$,
we infer from \eqref{eq:psi-rec2}, and \eqref{eq:psi-pat} for $n-1$ by assumption that
$n \, \psi_{n-1}$ provides the term $n! \, x_1$ in \eqref{eq:psi-pat}. 

It remains the case $k=1$ of \eqref{eq:bnk-psi}. With ${\norm{\psi_n} = n}$ 
and using the same arguments for $\BE_{n,k}$ from above, we finally derive 
that \eqref{eq:bnk-leq} also holds for $k=1$, showing \eqref{eq:psi-ord} completely.
This completes the induction and finishes the proof.
\end{proof}

\begin{corollary} \label{cor:psi-num}
For $n \geq 1$, we have $\# \psi_n \leq \PFS(n)$.
\end{corollary}

\begin{proof}
Let $1 \leq j \leq n$.
By Lemma~\ref{lem:sigma-pi}, each $\sigpi_{j}^\star$ consists of monomials $x_\gamma$ 
with ${\gamma \in \PF_j}$, and $\# \sigpi_{j}^\star = \PF(j)$.
Hence, we have a decomposition by the partitions $\PF_j$ such that
\[
  \# \sum_{j=1}^{n} \sigpi_{j}^\star = \sum_{j=1}^{n} \# \sigpi_{j}^\star = \PFS(n).
\] 
Since Theorem~\ref{thm:psi-ord} shows that $\norm{\Psi_j} \leq j$, 
we infer from \eqref{eq:psi-rec3} that $\# \psi_n \leq \PFS(n)$.
\end{proof}

\begin{proof}[Proof of Theorem~\ref{thm:num}]
This follows from Corollary~\ref{cor:psi-num}.  
\end{proof}

\begin{remark}
Regarding Theorem \ref{thm:psi-ord}, we should have sharper statements such that 
\begin{equation} \label{eq:psi-bnk-ord}
  \norm{\Psi_n} = n \andq \norm{\BE_{n,k}(\psi_1,\ldots,\psi_{n-k+1})} = n
\end{equation}
for ${n \geq 2}$ and ${1 \leq k < n}$, which seem to be supported by Tables~\ref{tbl:Psi2} 
and~\ref{tbl:Bnk}, respectively. However, since terms of the polynomials $\psi_\nu$ have 
different signs (see Tables~\ref{tbl:psi} and~\ref{tbl:psi2}), 
terms may be canceled out when computing \eqref{eq:psi-bnk-ord}. 
For the case $\Psi_n$, e.g., compare Tables~\ref{tbl:Psi} and~\ref{tbl:Psi2}.
For the case $\BE_{n,k}$, one may conjecture in view of Table~\ref{tbl:Bnk} and further 
computed terms that $\BE_{n,k}$ always contains the term $(-1)^{n-k} \, \ST_2(n,k) \, x_1^n$.
\end{remark}

At the end, we consider the sum of the coefficients of the polynomials $\Psi_j$ and $\psi_j$ 
regarding Corollary~\ref{cor:psi-sum}.
The sequence of $\Psi_j(1,\ldots,1)$ begins
\[
  0, -2, 3, -16, 50, -366, 1932, -16\,640, 131\,112, -1\,272\,600, 13\,642\,200, \ldots,
\]
which is not yet contained in the OEIS~\cite{OEIS}. With the latter sequence, 
we compute by \eqref{eq:psi-rec4} the sequence of $\psi_j(1,\ldots,1)$ as
\[
  1, 0, 3, -4, 30, -186, 630, -11\,600, 26\,712, -1\,005\,480, 2\,581\,920, \ldots
\]
It appears that this sequence above probably corresponds to sequence \seqnum{A347978}, 
but with opposite sign. Similarly, the sequence of $-\psi_j(-1,\ldots,-1)$ reads
\[
  1, 2, 9, 44, 290, 2154, 19\,026, 186\,752, 2\,070\,792, 25\,119\,720, \ldots,
\]
which probably coincides with sequence \seqnum{A073478}.

Define the alternating harmonic numbers for $n \geq 1$ by 
\[
  \AH_n = \sum_{\nu=1}^{n} \frac{(-1)^{\nu+1}}{\nu}. 
\]
Supported by further computations, we arrive at the following conjecture.

\begin{conjecture}
For $n \geq 1$, we have 
\[
  \psi_n(\pm 1, \ldots, \pm 1) = - \BE_n(\mp \AH_1, \mp 2! \, \AH_2, \ldots, \mp n! \, \AH_n),
\] 
and the generating function is given by
\[
  \sum_{n \geq 1} \psi_n(\pm 1, \ldots, \pm 1) \frac{x^n}{n!}
  = 1 - (x+1)^{\mp 1/(1-x)}, 
\]
choosing the corresponding signs, respectively.
\end{conjecture}

\newpage


\appendix
\section{Computations}

\begin{table}[H] \small
\setstretch{1.25}
\begin{center}
\begin{tabular}{r@{\;=\;}l}
  \toprule
  $\BE_{1,1}$ & $x_1$ \\
  \midrule
  $\BE_{2,1}$ & $2 x_1 - x_1^2 - x_2$ \\
  $\BE_{2,2}$ & $x_1^2$ \\
  \midrule
  $\BE_{3,1}$ & $6 x_1 - 6 x_1^2 + x_1^3  + 3 x_1 x_2 - 3 x_2 + 2 x_3$ \\
  $\BE_{3,2}$ & $6 x_1^2 - 3 x_1^3 - 3 x_1 x_2$ \\
  $\BE_{3,3}$ & $x_1^3$ \\
  \midrule
  $\BE_{4,1}$ & $24 x_1 - 36 x_1^2 + 12 x_1^3 - x_1^4 - 6 x_1^2 x_2 + 24 x_1 x_2 - 8 x_1 x_3 - 12 x_2 - 3 x_2^2 + 8 x_3 - 6 x_4$ \\
  $\BE_{4,2}$ & $36 x_1^2 - 36 x_1^3 + 7 x_1^4 + 18 x_1^2 x_2 - 24 x_1 x_2 + 8 x_1 x_3 + 3 x_2^2$ \\
  $\BE_{4,3}$ & $12 x_1^3 - 6 x_1^4 - 6 x_1^2 x_2$ \\
  $\BE_{4,4}$ & $x_1^4$ \\
  \bottomrule
\end{tabular}

\caption{First few polynomials $\BE_{n,k}(\psi_1,\ldots,\psi_{n-k+1})$.}
\label{tbl:Bnk}
\end{center}
\end{table}

\begin{table}[H] \small
\setstretch{1.25}
\begin{center}
\begin{tabular}{r@{\;=\;}l}
  \toprule
  $\Psi_1$ & $0$ \\
  $\Psi_2$ & $-2 \BE_{2,2}$ \\
  $\Psi_3$ & $9 \BE_{2,2} - 2 \BE_{3,2} - 6 \BE_{3,3}$ \\
  $\Psi_4$ & $-12 \BE_{2,2} + 12 \BE_{3,2} + 44 \BE_{3,3} - 2 \BE_{4,2} - 6 \BE_{4,3} - 24 \BE_{4,4}$ \\
  $\Psi_5$ & $-20 \BE_{3,2} - 120 \BE_{3,3} + 15 \BE_{4,2} + 55 \BE_{4,3} + 250 \BE_{4,4} - 2 \BE_{5,2} - 6 \BE_{5,3} - 24 \BE_{5,4} - 120 \BE_{5,5}$ \\
  \bottomrule
\end{tabular}

\caption{First few polynomials $\Psi_j$ in terms of $\BE_{n,k}(\psi_1,\ldots,\psi_{n-k+1})$.}
\label{tbl:Psi}
\end{center}
\end{table}

\begin{table}[H] \small
\setstretch{1.25}
\begin{center}
\begin{tabular}{r@{\;=\;}l}
  \toprule
  $\Psi_1$ & $0$ \\
  $\Psi_2$ & $-2 x_1^2$ \\
  $\Psi_3$ & $-3 x_1^2 + 6 x_1 x_2$ \\
  $\Psi_4$ & $-12 x_1^2 + 8 x_1^3 -2 x_1^4 + 12 x_1 x_2 - 16 x_1 x_3 - 6 x_2^2$ \\
  $\Psi_5$ & $-60 x_1^2 + 60 x_1^3 - 15 x_1^4 + 20 x_1^3 x_2 - 60 x_1^2 x_2 + 60 x_1 x_2 - 40 x_1 x_3 + 60 x_1 x_4 - 15 x_2^2 + 40 x_2 x_3$ \\
  \bottomrule
\end{tabular}

\caption{First few polynomials $\Psi_j$ in terms of $x_k$.}
\label{tbl:Psi2}
\end{center}
\end{table}

\begin{table}[H] \small
\setstretch{1.25}
\begin{center}
\begin{tabular}{r@{\;}c@{\;}l}
  \toprule
  $\psi_5$ & $=$ & $120 x_1 - 240 x_1^2 + 120 x_1^3 - 20 x_1^4 + x_1^5 + 10 x_1^3 x_2 - 90 x_1^2 x_2 + 20 x_1^2 x_3 + 180 x_1 x_2$ \\
  & & $+ 15 x_1 x_2^2 - 80 x_1 x_3 + 30 x_1 x_4 - 60 x_2 - 30 x_2^2 + 20 x_2 x_3 + 40 x_3 - 30 x_4 + 24 x_5$ \\
  $\psi_6$ & $=$ & $720 x_1 - 1800 x_1^2 + 1200 x_1^3 - 300 x_1^4 + 30 x_1^5 - x_1^6  
  - 15 x_1^4 x_2 + 240 x_1^3 x_2 - 40 x_1^3 x_3$ \\
  & & $- 1080 x_1^2 x_2 - 45 x_1^2 x_2^2 + 360 x_1^2 x_3 - 90 x_1^2 x_4 + 1440 x_1 x_2 + 270 x_1 x_2^2 - 120 x_1 x_2 x_3$ \\
  & & $- 720 x_1 x_3 + 360 x_1 x_4 - 144 x_1 x_5 - 360 x_2 - 270 x_2^2 - 15 x_2^3 + 240 x_2 x_3 - 90 x_2 x_4$ \\
  & & $+ 240 x_3 - 40 x_3^2 - 180 x_4 + 144 x_5 - 120 x_6$ \\
  \bottomrule
\end{tabular}

\caption{Multivariate polynomials $\psi_j$ continued.}
\label{tbl:psi2}
\end{center}
\end{table}


\bibliographystyle{amsplain}

\end{document}